\documentclass{amsart}

\usepackage{amsmath, amssymb, amscd, mathrsfs, bm}

\usepackage{xcolor}
\usepackage[linktocpage=true,bookmarksopen=true]{hyperref}
\usepackage[only,mapsfrom]{stmaryrd}

\newtheorem{theorem}{Theorem}[section]
\newtheorem{proposition}[theorem]{Proposition}

\theoremstyle{remark}
\newtheorem{remark}[theorem]{Remark}

\numberwithin{equation}{section}

\newcommand{\Order}{\mathcal{O}}
\newcommand{\into}{\hookrightarrow}

\newcommand{\onto}{\twoheadrightarrow}
\newcommand{\isomto}{\overset{\sim}{\to}}

\newcommand{\tensor}{\mathbin{\otimes}}
\newcommand{\closure}[1]{\overline{#1}}
\newcommand{\genby}[1]{\langle #1 \rangle}

\newcommand{\Z}{\mathbb{Z}}
\newcommand{\Q}{\mathbb{Q}}
\newcommand{\R}{\mathbb{R}}
\newcommand{\C}{\mathbb{C}}

\newcommand{\Gm}{\mathbf{G}_{m}}

\newcommand{\tor}{\mathrm{tor}}

\newcommand{\alg}[1]{\mathbf{#1}}
\newcommand{\var}{\;\cdot\;}
\newcommand{\ideal}[1]{\mathfrak{#1}}

\newcommand{\GL}{\mathit{GL}}

\DeclareMathOperator{\Hom}{Hom}

\DeclareMathOperator{\Spec}{Spec}

\DeclareMathOperator{\Pic}{Pic}

\DeclareMathOperator{\Lie}{Lie}

\DeclareFontFamily{U}{wncy}{}
\DeclareFontShape{U}{wncy}{m}{n}{<->wncyr10}{}
\DeclareSymbolFont{mcy}{U}{wncy}{m}{n}
\DeclareMathSymbol{\Sha}{\mathord}{mcy}{"58}

\newsavebox\myVerb 

\makeatletter
\@namedef{subjclassname@2020}{\textup{2020} Mathematics Subject Classification}
\makeatother

\title[Duality invariance of Faltings heights]
{Duality invariance of Faltings heights, Hodge line bundles and global periods}

\author{Takashi Suzuki}
\address{
	Department of Mathematics, Chuo University,
	1-13-27 Kasuga, Bunkyo-ku, Tokyo 112-8551, Japan
}
\email{tsuzuki@gug.math.chuo-u.ac.jp}

\date{September 3, 2025}

\begin{document}

\begin{abstract}
	We prove that an abelian variety and its dual over a global field
	have the same Faltings height
	and, more precisely, have isomorphic Hodge line bundles,
	including their natural metrized bundle structures.
	More carefully treating real places,
	we also show that these abelian varieties have
	the same real and global periods
	that appear in the Birch--Swinnerton-Dyer conjecture.
\end{abstract}

\maketitle

\tableofcontents


\section{Introduction}


\subsection{Main results}

Let $K$ be a global field.
Let $X$ be the $\Spec$ of the ring of integers of $K$
in the number field case
and the proper smooth curve with function $K$
in the function field case.
Let $A$ be an abelian variety over $K$ of dimension $g$
with N\'eron model $\mathcal{A}$ over $X$.
Let $\omega_{\mathcal{A}}$ be the Hodge line bundle of $\mathcal{A}$,
namely the pullback of $\Omega^{g}_{\mathcal{A} / X}$
by the zero section $X \into \mathcal{A}$.
It has a natural metrized line bundle structure
in the number field case (\cite[Section 3]{Fal86}).
Let $h(A)$ be the Faltings height of $A$,
which is the Arakelov degree of $\omega_{\mathcal{A}}$
in the number field case
and the usual degree of $\omega_{\mathcal{A}}$
in the function field case.
In this paper, we will prove the following
duality invariance of Faltings heights:

\begin{theorem}[Theorem \ref{0014}] \label{0010}
	Let $A$ and $B$ be abelian varieties over $K$
	dual to each other.
	Then $h(A) = h(B)$.
\end{theorem}

This is known in the number field case by Raynaud
(\cite[Corollary 2.1.3]{Ray85}).
In the function field case,
it is known when $A$ and $B$ have semistable reduction everywhere
by Moret-Bailly
(\cite[Chapter IX, Lemma 2.4]{MB85};
see also \cite[Remark 2.1.5]{Ray85}
and \cite[Proposition 2.16]{GLFP25}).
The function field, non-semistable case was open
as mentioned in \cite[Remark 2.1.5]{Ray85}
and \cite[Remark 2.17]{GLFP25}.
Our proof is a simple application of results of
the author's work \cite{GRS24} with Ghosh and Ray.

We will also prove a more precise result
using instead the author's other work \cite{OS23} with Overkamp.
To state our result,
first note that there is a canonical isomorphism
$\omega_{A} \cong \omega_{B}$
of the generic fibers of $\omega_{\mathcal{A}}$ and $\omega_{\mathcal{B}}$
given by (denoting dual by $\ast$)
	\begin{equation} \label{0013}
			(\det \Lie B)^{\ast}
		\cong
			(\det H^{1}(A, \Order))^{\ast}
		\cong
			H^{g}(A, \Order)^{\ast}
		\cong
			\Gamma(A, \Omega^{g}),
	\end{equation}
where the first isomorphism is by deformation,
the second by cup product
(\cite[Chapter VII, Section 4.21, Theorem 10]{Ser88})
and the third by Serre duality.

\begin{theorem}[Theorem \ref{0004}] \label{0011}
	Under the above isomorphism $\omega_{A} \cong \omega_{B}$,
	their line subbundles $\omega_{\mathcal{A}}$ and $\omega_{\mathcal{B}}$
	correspond to each other.
	In particular, we have
	$\omega_{\mathcal{A}} \cong \omega_{\mathcal{B}}$
	as $\Order_{X}$-modules.
\end{theorem}

This implies $h(A) = h(B)$ by taking the degrees
in the function field case,
giving a second (independent) proof of Theorem \ref{0010}.
A weaker form of the second sentence of Theorem \ref{0011} has been known:
\cite[Corollary 2.1.3]{Ray85} shows the existence of an isomorphism of the tensor squares
$\omega_{\mathcal{A}}^{\tensor 2} \cong \omega_{\mathcal{B}}^{\tensor 2}$
in the number field case%
\footnote{
	Raynaud's proof uses a relation between
	different ideals of isogenies and their duals
	(\cite[Theorem 2.1.1]{Ray85}).
	Over a function field, however,
	any isogeny of degree divisible by $p$ is either inseparable
	or with inseparable dual.
	The different ideal of a (generically) inseparable isogeny is zero
	by the formulas in \cite[(4.9.5), (4.9.6)]{Ill85}.
}
and \cite[Chapter IX, Lemma 2.4]{MB85} shows the existence of an isomorphism
$\omega_{\mathcal{A}}^{\tensor N} \cong \omega_{\mathcal{B}}^{\tensor N}$
for some $N \ge 1$
in the function field, semistable case.
Hence the second sentence of Theorem \ref{0011} is a tiny improvement
even in the number field case.
These being said,
the point of Theorem \ref{0011} is that
the isomorphism on the generic fibers is the canonical one \eqref{0013}.

See also Yuan's work \cite{Yua21} for this kind of study
of Hodge bundles themselves (rather than their degrees)
of abelian varieties over function fields
(though \cite{Yua21} does not discuss duality aspects
and it is more about the full Hodge bundle
$e^{\ast} \Omega^{1}_{\mathcal{A} / X}$
than its determinant).

We can further refine Theorem \ref{0011} in the number field case
taking metrics into account.
In this case, note that
$\omega_{\mathcal{A}}$ has another metrized bundle structure
given by integration over real points $A(\R)$
for each real place of $K$ (see \eqref{0019})
and integration over complex points $A(\C)$
for each complex place of $K$,
while Faltings's original metrized bundle structure uses
integration over complex points $A(\C)$
for all infinite places.
The Arakelov degree of the former metrized bundle structure is
the definition of the \emph{global period} of $A / K$
that appears in the right-hand side of the Birch--Swinnerton-Dyer conjecture
(\cite[Conjecture 2.1 (2)]{DD10}, \cite[Definition 2.1]{DD15}).
Let us call the former the \emph{BSD metrized bundle structure}
and the latter the \emph{Faltings metrized bundle structure}.

\begin{theorem}[Theorems \ref{0012}, \ref{0031}] \label{0029}
	Assume that $K$ is a number field.
	The isomorphism
	$\omega_{\mathcal{A}} \cong \omega_{\mathcal{B}}$
	in Theorem \ref{0011}
	preserves both the Faltings metrized bundle structures
	and the BSD metrized bundle structures.
	In particular, $A$ and $B$ have the same Faltings height
	and the same global period.
\end{theorem}

This gives another proof of $h(A) = h(B)$ in the number field case.
The duality invariance of global periods in Theorem \ref{0029} also follows from
Dokchitser--Dokchitser's unconditional isogeny invariance
of the BSD formula in \cite[Theorem 4.3]{DD10},
which is a global result.
In contrast, our Theorem \ref{0029} more precisely gives a comparison of periods
for each infinite place
and hence is purely local over $\C$ and $\R$.
Our proof does not even require the abelian varieties
to be definable over number fields.


\subsection{Outline of the proofs}

The proof of Theorem \ref{0010} first uses results of \cite{GRS24}
that say that $h(A_{1}) - h(A_{2})$
for any isogenous abelian varieties $A_{1}$ and $A_{2}$ is equal to
the difference of the dimensions (or the ``$\mu$-invariants'') of
the Tate--Shafarevich schemes of $A_{1}$ and $A_{2}$
in the sense of \cite{Suz20a}.
But these group schemes have the same dimension
if $A_{1}$ is dual to $A_{2}$
by the duality result in \cite{Suz20a}.

For Theorem \ref{0011},
first note that the statement of this theorem
(but a priori not of Theorem \ref{0010}) is purely local at finite places.
Hence we may replace $K$ by a non-archimedean local field.
The duality invariance $c(A) = c(B)$
of Chai's base change conductors (\cite[Section 1]{Cha00})
proved in \cite{OS23} reduces the statement to the case
where $A$ and $B$ have semistable reduction.
In the semistable case,
the Raynaud group scheme construction
(\cite[Chapter III, Theorem C.15]{Mil06})
essentially reduces the statement to the good reduction case.
The good case is just a relative version of \eqref{0013}.

We actually allow any perfect field of positive characteristic
as the constant field of the function field $K$.
Correspondingly, Theorem \ref{0010} is proved for
any proper smooth geometrically connected curve
over a perfect field of positive characteristic
and the local version of
Theorem \ref{0011} is proved for
any complete discrete valuation field
with perfect residue field of positive characteristic.

On the other hand, the statement and the proof of Theorem \ref{0029}
are purely about abelian varieties over $\C$ and $\R$
and hence essentially very classical.
Simple applications of Hodge theory and uniformization over $\C$
with some special care of complex conjugations over $\R$ are sufficient.

As a convention,
all group schemes in this paper are assumed commutative.


\subsection{Note on concurrent work}

After a preprint version of this paper was uploaded to arXiv,
the author learned that Vologodsky also obtained
Theorem \ref{0011}, the duality invariance of base change conductors
(both only for the finite residue field case)
and Theorem \ref{0029}
independently at more or less the same time.%
\begin{lrbox}\myVerb%
    \footnotesize\verb|https://researchseminars.org/talk/MITNT/71/|%
\end{lrbox}%
\footnote{
    The abstract and the video of Vologodsky's talk at MIT
    on March 8th, 2023 titled
    ``Dual abelian varieties over a local field have equal volumes''
    are available at:
    \par\noindent
    \usebox\myVerb
}
His proof for Theorem \ref{0011} and $c(A) = c(B)$ uses
Tate--Milne's duality (for finite residue fields)
while our proof uses the author's duality \cite{Suz20a} and \cite{Suz20b}
(for general perfect residue fields) among other similarities and differences.


\subsection{Acknowledgments}

The author is grateful to
Tim Dokchitser, Vladimir Dokchitser,
Richard Griffon, Samuel Le Fourn, Fabien Pazuki,
Jishnu Ray and the referees
for helpful discussions and comments.
Special thanks to Vadim Vologodsky for the information about his work.


\section{Heights}

We recall the constructions in \cite[Sections 4, 5, 8 and 9]{GRS24}.
Let $X$ be an irreducible quasi-compact regular scheme of dimension $1$
with perfect residue field of positive characteristic at closed points.
Let $K$ be its function field.
Let $A_{1}$ and $A_{2}$ be abelian varieties over $K$
with N\'eron models $\mathcal{A}_{1}$ and $\mathcal{A}_{2}$,
respectively, over $X$.
Let $f \colon A_{1} \to A_{2}$ be an isogeny over $K$.
Let $\mathcal{N} \subset \mathcal{A}_{1}$ be
the schematic closure of the kernel of $f$ in $\mathcal{A}_{1}$,
which is a quasi-finite flat separated group scheme over $X$.
Let $R \Lie \mathcal{N} = l^{\vee}_{\mathcal{N}} \in D^{b}(\Order_{X})$ be
its Lie complex (\cite[Chapter VII, Section 3.1.1]{Ill72}),
which is (represented by) a perfect complex of $\Order_{X}$-modules.
Let $\det R \Lie \mathcal{N}$ be its determinant invertible sheaf
(\cite[Tag 0FJW]{Sta25}).

The fppf quotient $\mathcal{A}_{2}' := \mathcal{A}_{1} / \mathcal{N}$
exists as a quasi-compact smooth separated group scheme over $X$
by \cite[Theorem 4.C]{Ana73}.
We have the induced morphism
$\mathcal{A}_{2}' \to \mathcal{A}_{2}$
and hence a diagram with exact rows
	\begin{equation} \label{0006}
		\begin{CD}
				0
			@>>>
				\mathcal{N}
			@>>>
				\mathcal{A}_{1}
			@>>>
				\mathcal{A}_{2}'
			@>>>
				0
			\\ @. @. @. @VVV \\
			@. @. @. \mathcal{A}_{2}.
		\end{CD}
	\end{equation}
Since $\mathcal{A}_{2}' \to \mathcal{A}_{2}$ is
an isomorphism on the generic fibers,
the induced morphism
$\det \Lie \mathcal{A}_{2}' \into \det \Lie \mathcal{A}_{2}$
on the determinant invertible sheaves of the Lie algebras
determines a unique effective divisor $\ideal{c}(f)$ on $X$
such that
	\[
			\det \Lie \mathcal{A}_{2}' \tensor \Order_{X}(\ideal{c}(f))
		\isomto
			\det \Lie \mathcal{A}_{2}.
	\]
This divisor $\ideal{c}(f)$ is called
the \emph{conductor divisor} of $f$
(\cite[Definition 5.1 (1)]{GRS24}).

\begin{proposition} \label{0007}
	We have a canonical isomorphism
		\[
					\det \Lie \mathcal{A}_{1}
				\tensor
					(\det \Lie \mathcal{A}_{2})^{\tensor -1}
			\cong
					\det R \Lie \mathcal{N}
				\tensor
					\Order_{X}(- \ideal{c}(f))
		\]
	of $\Order_{X}$-modules,
	where $(\var)^{\tensor -1}$ denotes the dual bundle.
\end{proposition}

\begin{proof}
	Applying $R \Lie$ to \eqref{0006},
	we have a diagram with distinguished triangle row
	\[
		\begin{CD}
				R \Lie \mathcal{N}
			@>>>
				\Lie \mathcal{A}_{1}
			@>>>
				\Lie \mathcal{A}_{2}'
			\\ @. @. @VVV \\
			@. @. \Lie \mathcal{A}_{2}
		\end{CD}
	\]
	by \cite[Chapter VII, Proposition 3.1.1.5]{Ill72}.
	Now apply $\det$.
\end{proof}

We recall the following result from \cite{GRS24}.
It will be used in the next section
(but not in this section).

\begin{proposition} \label{0021} \mbox{}
	\begin{enumerate}
	\item \label{0015}
		Let $\mathcal{A}$ and $\mathcal{B}$ be abelian schemes over $X$
		dual to each other.
		Then there exists a canonical isomorphism
		$\det \Lie \mathcal{A} \cong \det \Lie \mathcal{B}$
		of $\Order_{X}$-modules
		(which is the relative version of \eqref{0013}).
	\item \label{0016}
		Let $\mathcal{N}$ and $\mathcal{M}$ be
		finite flat group schemes over $X$
		Cartier dual to each other.
		Then there exists a canonical isomorphism
		$\det R \Lie \mathcal{N} \cong (\det R \Lie \mathcal{M})^{\tensor -1}$
		of $\Order_{X}$-modules.
	\item \label{0017}
		Let $0 \to \mathcal{N} \to \mathcal{A}_{1} \to \mathcal{A}_{2} \to 0$
		be an exact sequence of group schemes over $X$
		such that the $\mathcal{A}_{i}$ are abelian schemes
		and $\mathcal{N}$ finite flat.
		Let $0 \to \mathcal{M} \to \mathcal{B}_{2} \to \mathcal{B}_{1} \to 0$
		be its dual exact sequence.
		Then the diagram
			\[
				\begin{CD}
							\det \Lie \mathcal{A}_{1}
						\tensor
							(\det \Lie \mathcal{A}_{2})^{\tensor -1}
					@=
							\det \Lie \mathcal{B}_{1}
						\tensor
							(\det \Lie \mathcal{B}_{2})^{\tensor -1}
					\\ @| @| \\
						\det R \Lie \mathcal{N}
					@=
						(\det R \Lie \mathcal{M})^{\tensor -1}
				\end{CD}
			\]
		commutes,
		where the upper (respectively, lower) horizontal isomorphism is
		the isomorphism in \eqref{0015} (respectively, \eqref{0016})
		and the vertical isomorphisms are the natural ones.
	\item \label{0022}
		The isomorphisms in \eqref{0015} and \eqref{0016} are compatible
		with Zariski localization on $X$.
	\end{enumerate}
\end{proposition}

\begin{proof}
	This is \cite[Proposition 9.1]{GRS24}
	and its proof
	when $X$ is a smooth connected separated curve over a perfect field.
	The same proof works for general $X$.
\end{proof}

About \eqref{0015}, see also \cite[Lemma 1.1.3]{Lau96}
and \cite[Lemma (13.9)]{EGM12}.

Assume that $X$ is a proper smooth geometrically connected curve
over a perfect field $k$ of positive characteristic.
The degree $\deg \ideal{c}(f)$ of $\ideal{c}(f)$
with respect to the base field $k$
is called the \emph{conductor} of $f$ and denoted by $c(f)$
(\cite[Definition 5.1 (2)]{GRS24}).

\begin{proposition} \label{0008}
	We have
		\[
					\deg \Lie \mathcal{A}_{1}
				-
					\deg \Lie \mathcal{A}_{2}
			=
					\deg R \Lie \mathcal{N}
				-
					c(f).
		\]
\end{proposition}

\begin{proof}
	This follows from Proposition \ref{0007}.
\end{proof}

For $i = 1, 2$, let $\alg{H}^{1}(X, \mathcal{A}_{i})$ be
the Tate--Shafarevich group scheme over $k$ defined in \cite{Suz20a}.
It is the perfection (inverse limit along Frobenius)
of a smooth group scheme with unipotent identity component
(\cite[Theorem 3.4.1 (1), (2)]{Suz20a}).
Define $\mu_{A_{i} / K} = \dim \alg{H}^{1}(X, \mathcal{A}_{i})$
as in \cite[(6.1)]{GRS24}.

\begin{proposition} \label{0009}
	We have
		\[
				\mu_{A_{2} / K} - \mu_{A_{1} / K}
			=
					\deg \Lie \mathcal{A}_{1}
				-
					\deg \Lie \mathcal{A}_{2}.
		\]
\end{proposition}

\begin{proof}
	We have
		\[
				\mu_{A_{2} / K} - \mu_{A_{1} / K}
			=
					\deg R \Lie \mathcal{N}
				-
					c(f)
		\]
	by \cite[Theorem 10.3]{GRS24}.
	Hence the result follows from
	Proposition \ref{0008}.
\end{proof}

\begin{theorem} \label{0014}
	Let $A$ and $B$ be abelian varieties over $K$
	dual to each other.
	Then $\deg \Lie \mathcal{A} = \deg \Lie \mathcal{B}$.
\end{theorem}

\begin{proof}
	By \cite[Theorem 3.4.1 (6g) and Proposition 2.2.5]{Suz20a},
	the identity components of the perfect group schemes
	$\alg{H}^{1}(X, \mathcal{A})$ and
	$\alg{H}^{1}(X, \mathcal{B})$ are Breen--Serre dual
	(\cite[Chapter III, Lemma 0.13 (c)]{Mil06},
	\cite[Proposition 1.2.1]{Beg81}) to each other
	up to finite \'etale group schemes.
	But Breen--Serre duality does not change the dimension
	by \cite[Proposition 1.2.1]{Beg81}.
	Hence $\mu_{A / K} = \mu_{B / K}$.
	The result now follows from
	Proposition \ref{0009}.
\end{proof}

As $\omega_{\mathcal{A}}$ is dual to $\det \Lie \mathcal{A}$
and the same is true for $\mathcal{B}$,
Theorem \ref{0014} proves Theorem \ref{0010}
in the function field case.
The number field case is known
as noted after Theorem \ref{0010}.

\begin{remark}
	The formula in Proposition \ref{0009} is Iwasawa-theoretic
	as \cite{GRS24} itself is about
	Iwasawa theory for abelian varieties over function fields.
	It has an Iwasawa-theoretic analogue
	for elliptic curves over $\Q$
	by Dokchitser--Dokchitser (\cite[Theorem 1.1]{DD15}).
	Note that their formula contains an additional term
	that is the difference of Chai's base change conductors
	of the elliptic curves at $p$.
	If one wants to give a proof of $h(A) = h(B)$
	for abelian varieties $A, B$ over number fields dual to each other
	in the style of this section,
	one will have to extend \cite[Theorem 1.1]{DD15}
	to all isogenous abelian varieties over general number fields
	(possibly using the base change conductors
	of $A$ and $B$ at places above $p$).
\end{remark}


\section{Hodge line bundles}

Let $K$ be a complete discrete valuation field
with ring of integers $\Order_{K}$
and perfect residue field $k$ of positive characteristic.
Let $A_{1} \to A_{2}$ be an isogeny of abelian varieties over $K$
with kernel $N$.
Let $B_{2} \to B_{1}$ be its dual isogeny
with kernel $M$.
Assume that $A_{i}$ and $B_{i}$ have semistable reduction.
We have a canonical perfect pairing
$N \times_{K} M \to \Gm$
of finite group schemes over $K$.
Let $\mathcal{A}_{1} \to \mathcal{A}_{2}$
and $\mathcal{B}_{2} \to \mathcal{B}_{1}$ be
the induced morphisms on the N\'eron models
with kernels $\mathcal{N}$ and $\mathcal{M}$, respectively.
Then $\mathcal{N}$ and $\mathcal{M}$ are
quasi-finite flat separated group schemes over $\Order_{K}$
by \cite[Section 7.3, Lemma 5]{BLR90}.

\begin{proposition} \label{0000}
	The pairing
	$N \times_{K} M \to \Gm$
	over $K$ canonically extends to a pairing
	$\mathcal{N} \times_{\Order_{K}} \mathcal{M} \to \Gm$
	over $\Order_{K}$.
\end{proposition}

\begin{proof}
	We work in the derived category
	of fppf sheaves of abelian groups over $\Order_{K}$.
	Let $\mathcal{G}_{m}$ be the N\'eron (lft) model of $\Gm$.
	Let $\mathcal{C}$ be the mapping fiber of the morphism
		\[
				\mathcal{A}_{1} \tensor^{L} \mathcal{B}_{2}
			\to
					\mathcal{A}_{1} \tensor^{L} \mathcal{B}_{1}
				\oplus
					\mathcal{A}_{2} \tensor^{L} \mathcal{B}_{2}.
		\]
	We have canonical extensions
	$\mathcal{A}_{i} \tensor^{L} \mathcal{B}_{i} \to \mathcal{G}_{m}[1]$
	of the Poincar\'e biextension morphisms
	for $i = 1, 2$
	by \cite[Chapter III, Lemma C.10]{Mil06}.
	The morphisms for $i = 1$ and $2$ are compatible in the sense that the composites
		\[
				\mathcal{A}_{1} \tensor^{L} \mathcal{B}_{2}
			\to
				\mathcal{A}_{1} \tensor^{L} \mathcal{B}_{1}
			\to
				\mathcal{G}_{m}[1]
		\]
	and
		\[
				\mathcal{A}_{1} \tensor^{L} \mathcal{B}_{2}
			\to
				\mathcal{A}_{2} \tensor^{L} \mathcal{B}_{2}
			\to
				\mathcal{G}_{m}[1]
		\]
	are equal.
	Since any pairing
	$\mathcal{A}_{1} \times_{\Order_{K}} \mathcal{B}_{2} \to \mathcal{G}_{m}$
	is zero,
	the morphisms
	$\mathcal{A}_{i} \tensor^{L} \mathcal{B}_{i} \to \mathcal{G}_{m}[1]$
	uniquely come from a morphism
	$\mathcal{C} \to \mathcal{G}_{m}$.
	Composing it with the natural morphism
	$\mathcal{N} \tensor^{L} \mathcal{M} \to \mathcal{C}$,
	we obtain a pairing
	$\mathcal{N} \times \mathcal{M} \to \mathcal{G}_{m}$.
	It factors through the subgroup $\Gm$ of the target.
\end{proof}

\begin{proposition} \label{0032}
	Assume that $\mathcal{N}$ and $\mathcal{M}$ are finite over $\Order_{K}$.
	Then the pairing
	$\mathcal{N} \times_{\Order_{K}} \mathcal{M} \to \Gm$
	in Proposition \ref{0000} is perfect.
\end{proposition}

\begin{proof}
	For $i = 1, 2$,
	let $\mathcal{G}_{i}$ and $\mathcal{H}_{i}$ be
	the Raynaud group schemes for $A_{i}$ and $B_{i}$, respectively
	(\cite[Chapter III, Theorem C.15]{Mil06}).
	Let $\mathcal{T}_{i}$ and $\mathcal{S}_{i}$ be
	the torus parts of
	$\mathcal{G}_{i}^{0}$ and $\mathcal{H}_{i}^{0}$, respectively.
	Let $\mathcal{A}_{i}'$ and $\mathcal{B}_{i}'$ be
	the abelian scheme quotients of
	$\mathcal{G}_{i}^{0}$ and $\mathcal{H}_{i}^{0}$, respectively.
	Let $\mathcal{N}' =  \mathcal{N} \cap \mathcal{A}_{1}^{0}$ and
	$\mathcal{N}'' = \mathcal{N} \cap \mathcal{T}_{1}$.
	Let $\mathcal{M}'' \subset \mathcal{M}' \subset \mathcal{M}$ be similarly.
	Then $\mathcal{N}' / \mathcal{N}''$ and $\mathcal{M}' / \mathcal{M}''$
	are the kernels of the induced isogenies
	$\mathcal{A}_{1}' \onto \mathcal{A}_{2}'$
	and $\mathcal{B}_{2}' \onto \mathcal{B}_{1}'$.
	The restriction $\mathcal{N}' \times \mathcal{M}' \to \Gm$
	annihilates $\mathcal{N}''$ and $\mathcal{M}''$
	and the induced pairing
	$\mathcal{N}' / \mathcal{N}'' \times \mathcal{M}' / \mathcal{M}'' \to \Gm$
	is perfect as in the proof of \cite[Chapter III, Theorem C.15 (d)]{Mil06}.
	Consider the induced pairing
	$\mathcal{N}'' \times \mathcal{M} / \mathcal{M}' \to \Gm$.
	By assumption, $\mathcal{M} / \mathcal{M}'$ is finite.
	Hence this pairing is a paring between
	the finite flat multiplicative group scheme $\mathcal{N}''$
	and the finite \'etale group scheme $\mathcal{M} / \mathcal{M}'$.
	As it is a perfect pairing when base-changed to $K$
	as in the proof of \cite[Chapter III, Theorem C.15 (d)]{Mil06},
	it is perfect (over $\Order_{K}$).
	Similarly, the induced pairing
	$\mathcal{N} / \mathcal{N}' \times \mathcal{M}'' \to \Gm$
	is perfect.
	Hence $\mathcal{N} \times \mathcal{M} \to \Gm$ is perfect.
\end{proof}

\begin{proposition} \label{0001}
	Under the canonical isomorphism
	$\det R \Lie N \cong (\det R \Lie M)^{\tensor -1}$
	of $K$-vector spaces in Proposition \ref{0021} \eqref{0016},
	the $\Order_{K}$-lattices
	$\det R \Lie \mathcal{N}$ and $(\det R \Lie \mathcal{M})^{\tensor -1}$
	correspond to each other.
\end{proposition}

\begin{proof}
	Since $A_{i}$ and $B_{i}$ are semistable,
	the $\Order_{K}$-lattices in question
	are stable under base change $(\var) \tensor_{\Order_{K}} \Order_{L}$
	for any finite extension $L / K$.
	After such an extension,
	we may assume that $\mathcal{N}$ and $\mathcal{M}$ are
	finite over $\Order_{K}$
	by \cite[Section 2.3, Lemma 6]{Ber03}
	and hence Cartier dual to each other by
	Proposition \ref{0032}.
	This case is Proposition \ref{0021} \eqref{0022}.
\end{proof}

\begin{proposition} \label{0002}
	Under the canonical isomorphisms
		\[
				\det \Lie A_{1} \tensor (\det \Lie B_{1})^{\tensor -1}
			\cong
				\det \Lie A_{2} \tensor (\det \Lie B_{2})^{\tensor -1}
			\cong
				K
		\]
	of $K$-vector spaces in Proposition \ref{0021} \eqref{0015},
	the $\Order_{K}$-lattices
		\[
			\det \Lie \mathcal{A}_{1} \tensor (\det \Lie \mathcal{B}_{1})^{\tensor -1}
		\]
	and
		\[
			\det \Lie \mathcal{A}_{2} \tensor (\det \Lie \mathcal{B}_{2})^{\tensor -1}
		\]
	correspond to each other.
\end{proposition}

\begin{proof}
	This is a direct translation of
	Proposition \ref{0001}
	by Proposition \ref{0021} \eqref{0017}.
\end{proof}

\begin{proposition} \label{0003}
	Let $A$ and $B$ be any abelian varieties over $K$ dual to each other
	with N\'eron models $\mathcal{A}$ and $\mathcal{B}$, respectively.
	Under the canonical isomorphism
	$\det \Lie A \cong \det \Lie B$
	of $K$-vector spaces in Proposition \ref{0021} \eqref{0015},
	the $\Order_{K}$-lattices
	$\det \Lie \mathcal{A}$ and $\det \Lie \mathcal{B}$
	correspond to each other.
\end{proposition}

\begin{proof}
	First assume that $A$ and $B$ are semistable.
	Let $A \to B$ be any isogeny over $K$.
	Applying Proposition \ref{0002} to $A \to B$, we have
		\[
				\det \Lie \mathcal{A} \tensor (\det \Lie \mathcal{B})^{\tensor -1}
			=
				\det \Lie \mathcal{B} \tensor (\det \Lie \mathcal{A})^{\tensor -1}
			\subset
				K.
		\]
	Since the only fractional ideal $\ideal{p}_{K}^{n}$ (where $n \in \Z$)
	satisfying $\ideal{p}_{K}^{n} = \ideal{p}_{K}^{- n}$ in $K$
	is $\ideal{p}_{K}^{n} = \Order_{K}$ (so $n = 0$),
	we obtain the statement of the proposition in this case.
	
	For the general case,
	let $L$ be a finite Galois extension of $K$
	over which $A$ and $B$ have semistable reduction.
	Let $\mathcal{A}_{L}$ and $\mathcal{B}_{L}$ be
	the N\'eron models of $A \times_{K} L$ and $B \times_{K} L$,
	respectively.
	Then the semistable case implies that
	$\det \Lie \mathcal{A}_{L} \cong \det \Lie \mathcal{B}_{L}$
	over $\Order_{L}$ as lattices in
	$\det \Lie (A \times_{K} L) \cong \det \Lie (B \times_{K} L)$.
	We have natural inclusions
		\[
				(\det \Lie \mathcal{A}) \tensor_{\Order_{K}} \Order_{L}
			\into
				\det \Lie \mathcal{A}_{L},
		\]
		\[
				(\det \Lie \mathcal{B}) \tensor_{\Order_{K}} \Order_{L}
			\into
				\det \Lie \mathcal{B}_{L}
		\]
	of rank one free $\Order_{L}$-modules.
	The $\Order_{L}$-lengths of their cokernels are,
	by definition, $e_{L / K}$ times
	Chai's base change conductors (\cite[Section 1]{Cha00})
	of $A$ and $B$, respectively,
	where $e_{L / K}$ denotes the ramification index of $L / K$.
	But $A$ and $B$ have the same base change conductors
	by \cite[Theorem 1.2]{OS23}.
	Hence
		\[
				(\det \Lie \mathcal{A}) \tensor_{\Order_{K}} \Order_{L}
			\cong
				(\det \Lie \mathcal{B}) \tensor_{\Order_{K}} \Order_{L}
		\]
	in $\det \Lie (A \times_{K} L) \cong \det \Lie (B \times_{K} L)$.
	Therefore
	$\det \Lie \mathcal{A} \cong \det \Lie \mathcal{B}$
	in $\det \Lie A \cong \det \Lie B$, as desired.
\end{proof}

Let $X$ be an irreducible quasi-compact regular scheme of dimension $1$
with perfect residue field of positive characteristic at closed points.
Let $K$ be its function field.
Let $A$ and $B$ be abelian varieties over $K$
dual to each other.
Let $\mathcal{A}$ and $\mathcal{B}$ be their N\'eron models over $X$.

\begin{theorem} \label{0004}
	Under the canonical isomorphism
	$\det \Lie A \cong \det \Lie B$
	of $K$-vector spaces in Proposition \ref{0021} \eqref{0015},
	the line bundles
	$\det \Lie \mathcal{A}$ and $\det \Lie \mathcal{B}$
	correspond to each other.
\end{theorem}

\begin{proof}
	This follows from Proposition \ref{0003}.
\end{proof}

This proves Theorem \ref{0011}.


\section{Complex periods}

Let $A$ be an abelian variety over $\C$
of dimension $g$.
Let $\omega_{A}$ be the dual of $\det \Lie A$.
It is equipped with a hermitian metric given by
	\begin{equation} \label{0018}
			||\omega||^{2}
		=
			C(g)
			\left|
				\int_{A(\C)} \omega \wedge \closure{\omega}
			\right|
	\end{equation}
for $\omega \in \omega_{A}$,
where $C(g) \in \R_{> 0}$ is a choice of some normalization constant
that depends only on the integer $g = \dim A$
and not on $\omega$.

\begin{proposition} \label{0005}
	Let $B$ be the dual of $A$.
	The canonical isomorphism $\omega_{A} \cong \omega_{B}$
	of $\C$-vector spaces in \eqref{0013}
	preserves the hermitian metrics.
\end{proposition}

\begin{proof}
	Take a uniformization exact sequence
		\begin{equation} \label{0023}
				0
			\to
				\Z^{2 g}
			\stackrel{M}{\to}
				\R^{2 g}
			=
				\C^{g}
			\to
				A(\C)
			\to
				0
		\end{equation}
	with $M \in \GL_{g}(\C) \backslash \GL_{2 g}(\R) / \GL_{2 g}(\Z)$.
	Let $(z_{n})_{n = 1}^{g} = (x_{n} + i y_{n})_{n = 1}^{g}$ be
	the coordinates of $\C^{g} = \R^{2 g}$.
	We have a non-zero element $d z_{1} \wedge \dots \wedge d z_{g}$ of $\omega_{A}$,
	with
		\[
				||d z_{1} \wedge \dots \wedge d z_{g}||^{2}
			=
				C(g) 2^{g} |\det M|.
		\]
	We also have an exact sequence
		\begin{equation} \label{0024}
				0
			\to
				\Z^{2 g}
			\stackrel{(M^{\mathrm{T}})^{-1}}{\longrightarrow}
				\R^{2 g}
			=
				\C^{g}
			\to
				B(\C)
			\to
				0,
		\end{equation}
	where $M^{\mathrm{T}}$ is the transpose of $M$.
	Let $(z_{n}^{\ast})_{n = 1}^{g}$ be the coordinates
	of $\C^{g}$ in this sequence.
	We have a non-zero element
	$d z_{1}^{\ast} \wedge \dots \wedge d z_{g}^{\ast}$
	of $\omega_{B}$, with
		\[
				||d z_{1}^{\ast} \wedge \dots \wedge d z_{g}^{\ast}||^{2}
			=
				C(g) 2^{g} |\det M|^{-1}.
		\]
	
	Hence it is enough to show that
	the isomorphism $\omega_{A} \cong \omega_{B}$ gives the correspondence
		\begin{equation} \label{0026}
				d z_{1} \wedge \dots \wedge d z_{g}
			\leftrightarrow
				(-1)^{g (g + 1) / 2} (\det M)
				d z_{1}^{\ast} \wedge \dots \wedge d z_{g}^{\ast}.
		\end{equation}
	(The sign is not important here.)
	For clarity, we use the symbols
	$\Lambda = \Z^{2 g}$ and $V = \C^{g}$
	for the uniformization data for $A$ (not $B$).
	By \cite[Proposition 2.4.1]{BL04},
	the isomorphism $B(\C) \cong \Pic^{0}(A)$
	is given by the map on the group cohomology groups
		\begin{equation} \label{0025}
				\exp(2 \pi i \var)
			\colon
				H^{1}(\Lambda, \R / \Z)
			\stackrel{\sim}{\to}
				H^{1}(\Lambda, \Gamma(V, \Order^{\times})),
		\end{equation}
	where $\Order$ on the right is the holomorphic structure sheaf.
	Hence the isomorphism
	$\Lie B \cong H^{1}(A, \Order)$ can be written as
		\[
				2 \pi i
			\colon
				H^{1}(\Lambda, \R)
			\isomto
				H^{1}(\Lambda, \Gamma(V, \Order)).
		\]
	Hence the basis of $\Lie B$ dual to $(d z_{1}^{\ast}, \dots, d z_{g}^{\ast})$
	corresponds to the anti-holomorphic forms
	$(\pi d \closure{z_{1}}, \dots, \pi d \closure{z_{g}})$ in $H^{1}(A, \Order)$.
	Therefore the element of $\det \Lie B$ dual to
	$d z_{1}^{\ast} \wedge \dots \wedge d z_{g}^{\ast}$
	corresponds to the element
	$\pi^{g} d \closure{z_{1}} \wedge \dots \wedge d \closure{z_{g}}$
	of $H^{g}(A, \Order)$.
	Under the Serre duality
	$H^{g}(A, \Order) \leftrightarrow \Gamma(A, \Omega^{g}) = \omega_{A}$,
	the pairing between
	$\pi^{g} d \closure{z_{1}} \wedge \dots \wedge d \closure{z_{g}}$
	and $d z_{1} \wedge \dots \wedge d z_{g}$
	is given by
		\[
				\frac{1}{(2 \pi i)^{g}}
				\int_{A(\C)}
						d z_{1} \wedge \dots \wedge d z_{g}
					\wedge
						\pi^{g} d \closure{z_{1}} \wedge \dots \wedge d \closure{z_{g}}
			=
				(-1)^{g (g + 1) / 2}
				\det M,
		\]
	as desired.
\end{proof}

\begin{theorem} \label{0012}
	Let $A$ and $B$ be abelian varieties over a number field $K$
	dual to each other.
	Let $\omega_{\mathcal{A}}$ and $\omega_{\mathcal{B}}$ be
	the Hodge line bundles for $A$ and $B$, respectively.
	The isomorphism
	$\omega_{\mathcal{A}} \cong \omega_{\mathcal{B}}$
	in Theorem \ref{0011}
	preserves the Faltings metrized bundle structures.
\end{theorem}

\begin{proof}
	This follows from Proposition \ref{0005}.
\end{proof}

This proves the part of Theorem \ref{0029}
for the Faltings metrized bundle structures.


\section{Real periods}

Let $A$ be an abelian variety over $\R$.
Then $\omega_{A}$ can also be equipped with a Riemannian metric
given by
	\begin{equation} \label{0019}
			||\omega||
		=
			\int_{A(\R)} |\omega|
	\end{equation}
for $\omega \in \omega_{A}$.

\begin{proposition} \label{0020}
	Let $B$ be dual to $A$.
	Then the isomorphism $\omega_{A} \cong \omega_{B}$
	preserves this metrics.
\end{proposition}

\begin{proof}
	Write $A(\C) = V / \Lambda$,
	where $V$ is a $g$-dimensional $\C$-vector space
	and $\Lambda$ a rank $2 g$ lattice.
	Let $c \colon A(\C) \isomto A(\C)$ be the complex conjugation
	acting on the coefficient field $\C$,
	which induces an automorphism on $\Lambda$
	and a $\C$-semi-linear automorphism on $V$.
	Let $\Lambda^{c = 1}$ and $\Lambda_{c = 1}$ be
	the kernel and cokernel, respectively,
	of the endomorphism $c - 1$ on $\Lambda$.
	Let $\Lambda^{c = - 1}$ and $\Lambda_{c = - 1}$ be
	similarly of $c + 1$ on $\Lambda$.
	Let $V^{c = \pm 1}$ and $V_{c = \pm 1}$ be similarly for $V$.
	Let $(z_{1}, \dots, z_{g})$, $(w_{1}, \dots, w_{g})$
	and $(\lambda_{1}, \dots, \lambda_{2 g})$ be
	$\Z$-bases of $\Lambda^{c = 1}$, $\Lambda^{c = -1}$
	and $\Lambda$ respectively.
	Take $(\lambda_{1}, \dots, \lambda_{2 g})$
	and $(z_{1}, \dots, z_{g})$ to be
	the $\Z$-basis and the $\C$-basis, respectively,
	of $\Lambda$ and $V$, respectively.
	We consider the exact sequences \eqref{0023} and \eqref{0024}
	with respect to these bases.
	Write
		\[
				(w_{1}, \dots, w_{g})
			=
				P + i Q
		\]
	in $V^{g}$ with $P \in M_{g}(\R)$ and $Q \in \GL_{g}(\R)$.
	The action of the complex conjugation $c$ on $B(\C)$
	induces actions on the terms $\Z^{2 g}$ and $\R^{2 g}$ in \eqref{0024}.
	By \eqref{0025},
	these actions are given by
	$- c$ on $\Z^{2 g} = \Hom_{\Z}(\Lambda, \Z)$
	and $\R^{2 g} = \Hom_{\R}(V, \R)$.
	Let $(w_{1}^{\ast}, \dots, w_{n}^{\ast})$ be
	the basis of $\Hom_{\Z}(\Lambda^{c = - 1}, \Z)$
	dual to $(w_{1}, \dots, w_{g})$.
	
	Taking the kernel of $c - 1$ on \eqref{0023},
	we have an exact sequence
		\[
				0
			\to
				\bigoplus_{n = 1}^{g}
					\Z z_{n}
			\to
				\bigoplus_{n = 1}^{g}
					\R z_{n}
			\to
				A(\R)^{0}
			\to
				0,
		\]
	where $A(\R)^{0} \subset A(\R)$ is the identity component.
	Hence
		\[
				\int_{A(\R)^{0}}
					d z_{1} \wedge \dots \wedge d z_{g}
			=
				1,
		\]
	so
		\begin{equation} \label{0027}
				\int_{A(\R)}
					d z_{1} \wedge \dots \wedge d z_{g}
			=
				\# \pi_{0}(A(\R)).
		\end{equation}
	Taking the kernel of $c - 1$ on \eqref{0024},
	we have an exact sequence
		\[
				0
			\to
				\Hom_{\Z}(\Lambda_{c = -1} / \tor, \Z)
			\to
				\Hom_{\R}(V_{c = - 1}, \R)
			\to
				B(\R)^{0}
			\to
				0,
		\]
	where $(\var) / \tor$ denotes the torsion-free quotient.
	Define a real torus $T$ by the exact sequence
		\[
				0
			\to
				\bigoplus_{n = 1}^{g}
					\Z w_{n}^{\ast}
			\to
				\bigoplus_{n = 1}^{g}
					\R w_{n}^{\ast}
			\to
				T
			\to
				0.
		\]
	The map $c - 1 \colon \Lambda_{c = -1} \to \Lambda^{c = -1}$
	induces a commutative diagram with exact rows and columns
		\[
			\begin{CD}
				@.
					0
				@.
					0
				\\ @. @VVV @VVV \\
					0
				@>>>
					\bigoplus_{n = 1}^{g}
						\Z w_{n}^{\ast}
				@>>>
					\bigoplus_{n = 1}^{g}
						\R w_{n}^{\ast}
				@>>>
					T
				@>>>
					0
				\\ @. @VVV @VV \wr V @VVV \\
					0
				@>>>
					\Hom_{\Z}(\Lambda_{c = -1} / \tor, \Z)
				@>>>
					\Hom_{\R}(V_{c = - 1}, \R)
				@>>>
					B(\R)^{0}
				@>>>
					0
				\\ @. @. @VVV @VVV \\
				@.
				@.	
					0
				@.
					0.
			\end{CD}
		\]
	The cokernel of the left vertical arrow
	or, equivalently, the kernel of the right vertical arrow
	is isomorphic to the Pontryagin dual of
	$\Lambda^{c = -1} / (c - 1) \Lambda$.
	We have
		\[
				\Lambda^{c = -1} / (c - 1) \Lambda
			\cong
				\Hat{H}^{1}(\genby{c}, \Lambda)
			\cong
				\Hat{H}^{0}(\genby{c}, A(\C))
			\cong
				\pi_{0}(A(\R))
		\]
	by \cite[Chapter I, Remark 3.7]{Mil06},
	where $\Hat{H}$ denotes Tate cohomology.
	Let $d w_{1}, \dots, d w_{g}$ be the differential forms on
	$T$ or $B(\R)^{0}$ corresponding to $w_{1}, \dots, w_{g}$.
	Then
		\[
				\int_{T}
					d w_{1} \wedge \dots \wedge d w_{g}
			=
				1
		\]
	and hence
		\[
				\int_{B(\R)^{0}}
					d w_{1} \wedge \dots \wedge d w_{g}
			=
				\frac{1}{\# \pi_{0}(A(\R))}.
		\]
	Both $(d w_{1}, \dots, d w_{g})$ and
	$(d z_{1}^{\ast}, \dots, d z_{g}^{\ast})$ form $\C$-bases of $\Lie A(\C)$
	related by
		\[
				(d z_{1}^{\ast}, \dots, d z_{g}^{\ast})
			=
				(d w_{1}, \dots, d w_{g}) 2 Q^{-1}.
		\]
	Hence
		\[
				\int_{B(\R)^{0}}
					(\det M) d z_{1}^{\ast} \wedge \dots \wedge d z_{g}^{\ast}
			=
				\frac{2^{g} \det M}{\det Q \cdot \# \pi_{0}(A(\R))}.
		\]
	We have an exact sequence
		\[
				0
			\to
				\Lambda / (\Lambda^{c = 1} \oplus \Lambda^{c = -1})
			\stackrel{c - 1}{\to}
				\Lambda^{c = -1} / 2 (\Lambda^{c = -1})
			\to
				\Lambda^{c = -1} / (c - 1) \Lambda
			\to
				0.
		\]
	The middle term has $2^{g}$ elements
	and the right term has $\# \pi_{0}(A(\R))$ elements.
	Hence
		\begin{equation} \label{0030}
				\# \bigl(
					\Lambda / (\Lambda^{c = 1} \oplus \Lambda^{c = -1})
				\bigr)
			=
				\frac{2^{g}}{\# \pi_{0}(A(\R))},
		\end{equation}
	so
		\[
				\det Q
			=
				\frac{2^{g}}{\# \pi_{0}(A(\R))} \det M.
		\]
	Thus
		\[
				\int_{B(\R)^{0}}
					(\det M) d z_{1}^{\ast} \wedge \dots \wedge d z_{g}^{\ast}
			=
				1,
		\]
	so
		\begin{equation} \label{0028}
				\int_{B(\R)}
					(\det M) d z_{1}^{\ast} \wedge \dots \wedge d z_{g}^{\ast}
			=
				\pi_{0}(B(\R)).
		\end{equation}
	By \eqref{0026}, \eqref{0027} and \eqref{0028},
	we are reduced to showing that
		\[
				\# \pi_{0}(A(\R))
			=
				\# \pi_{0}(B(\R)).
		\]
	This follows from
		\[
				\# \bigl(
					\Lambda / (\Lambda^{c = 1} \oplus \Lambda^{c = -1})
				\bigr)
			=
				\frac{2^{g}}{\# \pi_{0}(B(\R))},
		\]
	which itself follows from \eqref{0030} with $c$ replaced by $- c$
	and \cite[Chapter I, Remark 3.7]{Mil06}.
\end{proof}

Let $A$ be an abelian variety over a number field $K$
with N\'eron model $\mathcal{A}$.
For a real place $v$ of $K$,
we give a Riemannian metric on $\omega_{A \times_{K} K_{v}}$
by \eqref{0019}.
For a complex place $v$ of $K$,
we give a Hermitian metric on $\omega_{A \times_{K} K_{v}}$
by \eqref{0018} with $C(g) = 1$.%
\begin{lrbox}\myVerb%
    \footnotesize\verb|https://virtualmath1.stanford.edu/~conrad/BSDseminar/Notes/L3.pdf|%
\end{lrbox}%
\footnote{
	The ``$2$'' and ``$2^{\dim A}$''
	in front of complex periods in
	\cite[Conjecture 2.1 (2)]{DD10}
	and \cite[Definition 2.1]{DD15}, respectively,
	should both be replaced by ``$1$''
	in accordance with Exercise 6.4 of Brian Conrad's seminar notes
	``N\'eron models, Tamagawa factors, and Tate--Shafarevich groups''
	available at:
    \par\noindent
    \usebox\myVerb
}
This defines a metrized bundle structure on $\omega_{\mathcal{A}}$,
which we call the \emph{BSD metrized bundle structure}.
Its degree is the definition of the \emph{global period} of $A / K$
as in \cite[Conjecture 2.1 (2)]{DD10} and \cite[Definition 2.1]{DD15}.

\begin{theorem} \label{0031}
	Let $B$ be dual to $A$.
	Then the isomorphism
	$\omega_{\mathcal{A}} \cong \omega_{\mathcal{B}}$
	in Theorem \ref{0011}
	preserves the BSD metrized bundle structures.
	In particular, $A$ and $B$ have the same global period.
\end{theorem}

\begin{proof}
	This follows from Proposition \ref{0020}.
\end{proof}

This proves the part of Theorem \ref{0029}
for the BSD metrized bundle structure,
finishing the proof of Theorem \ref{0029} itself.

\begin{remark}
	Proposition \ref{0005}
	(respectively, Proposition \ref{0020})
	is more generally true for complex tori $A$
	(respectively, complex tori $A$ with complex conjugation)
	by the same proof.
\end{remark}


\end{document}